\documentclass[notitlepage,11pt,reqno]{amsart}
\usepackage[letterpaper]{geometry}
\usepackage{amssymb}
\usepackage[breaklinks=true]{hyperref}

\newcommand{\Z}{\mathbb Z}
\newcommand{\F}{\mathbb F}
\newcommand{\I}{\mathcal I}
\newcommand{\calp}{\mathcal P}
\newcommand{\fk}{\mathcal{F}_k}
\newcommand{\W}{\mathbb W}
\newcommand{\m}{\mathfrak m}
\DeclareMathOperator{\Ann}{\ensuremath{Ann}}
\DeclareMathOperator{\frob}{\ensuremath{Frob}}

\newtheorem{theorem}[equation]{Theorem}
\newtheorem{prop}[equation]{Proposition}
\newtheorem{lemma}[equation]{Lemma}
\newtheorem{conj}[equation]{Conjecture}

\theoremstyle{definition}
\newtheorem{remark}[equation]{Remark}

\numberwithin{equation}{section}

\begin{document}
\title[A Carlitz--von Staudt theorem for finite rings]{A Carlitz--von
Staudt type theorem\\ for finite rings}

\author{Apoorva Khare}
\address[A.~Khare]{Indian Institute of Science;
Analysis and Probability Research Group; 
Bangalore, India}
\email{\tt khare@iisc.ac.in}

\author{Akaki Tikaradze}
\address[A.~Tikaradze]{Department of Mathematics, University of Toledo,
Toledo, USA}
\email{\tt tikar06@gmail.com}

\thanks{A.K.~is partially supported by a Ramanujan Fellowship (SERB) and
a Young Investigator Award from the Infosys Foundation. A.T.~was
partially supported by the University of Toledo Summer Research Awards
and Fellowships Program.}

\subjclass[2010]{Primary: 16P10; Secondary: 16K20}

\keywords{Staudt--Clausen, power sums, finite ring, matrix ring,
translation-invariant polynomial}

\date{\today}

\begin{abstract}
We compute the $k$th power-sums (for all $k>0$) over an arbitrary finite
unital ring $R$. This unifies and extends the work of Brawley, Carlitz,
and Levine for matrix rings [\textit{Duke Math.~J.}~1974], with folklore
results for finite fields and finite cyclic groups, and more general
recent results of Grau and Oller-Marc\'en for commutative rings
[\textit{Finite Fields Appl.}~2017].
As an application, we resolve a conjecture by Fortuny Ayuso, Grau,
Oller-Marc\'en, and R\'ua on zeta values for matrix rings over finite
commutative rings [\textit{Internat. J. Algebra Comput.}~2017].
We further recast our main result via zeta values over polynomial rings,
and end by classifying the translation-invariant polynomials over a large
class of finite commutative rings.
\end{abstract}

\maketitle

\section{Introduction and power-sum formulas}

Given a finite unital ring $R$, in this work we are interested in
computing the power-sums
\[
\zeta_R(-k) := \sum_{r \in R} r^k, \qquad k > 0.
\]

\noindent When $R = \Z / n \Z$, this involves computing the sum of the
first $n$ $k$th powers modulo $n$. Interest in this problem goes back at
least as far as the work of Clausen \cite{Cl} and von Staudt \cite{St};
clearly, the problem is also connected with Bernoulli numbers and
Bernoulli polynomials.

An important contribution to the area comes from the work of Carlitz on
zeta functions of $\F_q[T]$ and (rank one) Drinfeld modules. In
particular, Carlitz and his coauthors studied the above problem for more
general rings. We refer the reader to \cite{BCL,Ca2,Ca3} for more on
these themes.
For completeness we also remind that zeta values over finite fields,
\begin{equation}\label{Efield}
\zeta_{\F_q}(-k) = \begin{cases}
-1, \qquad & \text{if } (q-1)|k,\\
0,  \qquad & \text{otherwise},
\end{cases}
\end{equation}

\noindent have modern applications to error-correcting codes such as the
(dual) Reed--Solomon code, and more generally the BCH
(Bose--RayChaudhuri--Hocquenghem) code, via the discrete Fourier
transform.

The problem of computing zeta values for other finite (commutative) rings
has recently attracted renewed attention \cite{FGO,FGOR,GO,GOS}; see
also~\cite{Mor}. In particular, for commutative rings the question was
solved very recently in \cite{GO} (modulo a small typo in
their main theorem; see Remark \ref{Rtypo}).

In this paper, we are interested in the more challenging problem of
computing zeta values over non-commutative finite rings.
This problem was once again studied by Carlitz and his coauthors.
For instance, in \cite{BCL} the authors show:

\begin{theorem}[Brawley, Carlitz, and Levine, {\cite[Section
3]{BCL}}]\label{Tbcl}
If $R = M_{n \times n}(\F_q)$ for $n \geqslant 2$, and $k > 0$, then
\[
\zeta_R(-k) = \begin{cases}
{\rm Id}_{2 \times 2}, \qquad & \text{if } n = q = 2 \text{ and } 1 < k
\equiv -1,0,1 \hspace*{-3mm} \mod 6,\\
0, & \text{otherwise}.
\end{cases}
\]
\end{theorem}

Since the aforementioned work, to our knowledge the only other
non-commutative rings for which power-sums/zeta-values have been computed
to date, are matrix rings over $\Z/n\Z$ for various $n$.

In the present work, we compute all power sums over an arbitrary finite
ring $R$. To do so, notice that if $R$ has characteristic $\prod_{i=1}^l
p_i^{t_i}$ for pairwise distinct primes $p_i$, with $t_i>0\ \forall i$,
then $R = R_1 \times \cdots \times R_l$, where each ring $R_i$ has
characteristic $p_i^{t_i}$. Equipped with this notation, it is possible
to state the main result of the paper.

\begin{theorem}\label{Tmain}
Suppose as above that $R = \times_{i=1}^l R_i$ is an arbitrary finite
unital ring. Then:
\begin{equation}\label{Eproduct}
\zeta_R(-k) = \sum_{i=1}^l \frac{|R|}{|R_i|} \zeta_{R_i}(-k),
\end{equation}

\noindent where $k \in \mathbb{N}$, and the $i$th summand lives in $R_i$.

Now consider the problem for each factor $R = R_i$, with characteristic
equal to a power of a prime $p>0$. Let $T_{n \times n}(R')$ denote the
upper triangular $n \times n$ matrices over a ring $R'$, and let $x :=
\begin{pmatrix} 0 & 1\\ 0 & 0 \end{pmatrix}$. Then,
\begin{equation}
\zeta_R(-k) = \begin{cases}
\zeta_{\F_q}(-k), & \text{if } R = \F_q;\\
p^{m-1} \zeta_{\F_p}(-k), & \text{if } R = \Z / p^m \Z \text{ for } m
\geqslant 2, \text{ and }\\
& \qquad \qquad p>2 \text{ or } k=1 \text{ or } p = 2|k;\\
x , \ & \text{if } R = \F_2[x] / (x^2) \text{ or } T_{2 \times 2}(\F_2),
\text{ and } k > 1 \text{ odd};\\
{\rm Id}_{2 \times 2}, & \text{if } R = M_{2 \times 2}(\F_2) \text{ and }
1 < k \equiv 0, \pm 1 \mod 6;\\
0, & \text{otherwise}.
\end{cases}
\end{equation}

\noindent Here, $\zeta_{\F_q}(-k)$ is as in \eqref{Efield}, and $\F_2[x]
/ (x^2)$ is identified with the span of unipotent matrices in $T_{2
\times 2}(\F_q)$, with $1 \mapsto {\rm Id}_{2 \times 2}$.
\end{theorem}

In particular, Theorem \ref{Tmain} extends the aforementioned results,
and brings closure to the problem of computing zeta values over arbitrary
finite rings.\smallskip

We make a few observations about Theorem \ref{Tmain} before proceeding
further. First, when $R = \Z / p^m \Z$ in Theorem \ref{Tmain}, by
$\zeta_{\F_p}(-k)$ we mean any lift of this element to $\Z / p^m \Z$;
multiplying by $p^{m-1}$ gives independence from this choice of lift.
Moreover, the case of upper triangular $2 \times 2$ matrices can be
replaced by the appropriate assertion for lower triangular $2 \times 2$
matrices, or for any isomorphic ring $R$.

Second, note that the proof of Theorem \ref{Tmain} is more involved than
the corresponding result for $R$ commutative (which was shown in
\cite{GO}). It is interesting that the proof in the present work does not
make use of the findings of \cite{GO}, which instead fall out of our
proof as a consequence.

\begin{remark}\label{Rtypo}
We note a small typo in the main result of \cite{GO}. Specifically,
Theorem 1(iii) in it is not true as stated; for instance, if $R = R_1 =
\Z / 4\Z$ and $k>1$ is odd, then \cite[Theorem 1(iii)]{GO} says the zeta
value $\zeta_R(-k) = -(|R|/2) = 2 \mod 4$. However,
\[
\zeta_R(-k) \equiv 0^k + 1^k + 2^k + (-1)^k \equiv 0 \mod 4, \qquad
\text{for all } k > 1 \text{ odd}.
\]
The correction is that \cite[Theorem 1(iii)]{GO} holds for $k=1$, not for
$k>1$ odd. Indeed, the proof of \textit{loc.~cit.} uses \cite[Proposition
1(ii)]{GO}, which gives nonzero zeta values for $k=1$ but not for odd
$k>1$.
\end{remark}

\begin{remark}
We also make a small computational correction to \cite[Theorem 5.4]{BCL}
by Carlitz et al. The authors claimed that every power sum over $T_{n
\times n}(\F_q)$ vanishes for all $n,q \geqslant 2$. In particular, they
remarked that for $n=q=2$, ``a direct computation verifies the result''
(see the last line of their proof argument in \textit{loc.~cit.}). As we
state in our main theorem \ref{Tmain}, and verify in Proposition
\ref{Pcalc2} below, this computation actually shows that all odd-power
sums over $T_{2 \times 2}(\F_2)$ are nonzero, except for the sum of
elements of $T_{2 \times 2}(\F_2)$.
For all other (triangular) matrix rings over finite fields, $R = M_{n
\times n}(\F_q)$ or $T_{n \times n}(\F_q)$, the power sums were computed
by Carlitz et al in \cite[Section 3 and Theorem 5.4]{BCL}.
Akin to the contents of \cite{GO}, we do not make use of these results,
relying only on Theorem \ref{Tbcl}. The results in \cite{BCL} on zeta
values over triangular matrix rings also emerge out of our proof as
consequences.
\end{remark}

We conclude this section with an application of our main result, which
settles a conjecture in the recent paper \cite{FGOR}.

\begin{conj}[Fortuny Ayuso, Grau, Oller-Marc\'en, and R\'ua, 2015]
Let $d>1$ and $R$ be a finite (commutative) ring. Then all power sums
$\sum_{A \in M_{d \times d}(R)} A^k$ vanish unless the following
conditions hold:
\begin{enumerate}
\item $d=2$,
\item $|R| \equiv 2 \mod 4$ and $1 < k \equiv 0, \pm 1 \mod 6$,
\item The unique element $e \in R \setminus \{ 0 \}$ such that $2e=0$
satisfies: $e^2=e$.
\end{enumerate}
Moreover, in this case the aforementioned power sum equals $e \cdot {\rm
Id}_{2 \times 2}$.
\end{conj}

Recall that the authors stated two conjectures in \cite{FGOR}, first for
the commutative case and then for the non-commutative case. We prove the
latter, stronger version, for unital rings.

\begin{proof}
Write the finite unital ring $R = \times_{i=1}^l R_i$ as a product of
factor rings of distinct prime power characteristics $p_i^{t_i}$. Then,
\[
M_{d \times d}(R) = \times_{i=1}^l M_{d \times d}(R_i)
\]

\noindent is a product of non-commutative rings. By Theorem \ref{Tmain},
it follows that some $p_i = 2$ if there is a nonzero power sum. Let $p_1
= 2$ without loss of generality; then proving the conjecture for $R_1$
shows it for $R$, via Equation \eqref{Eproduct}. Now note that $R_1$ and
$M_{d \times d}(R_1)$ have the same characteristic. Applying Theorem
\ref{Tmain} and dimension considerations, $R_1 = \F_2$, and the
conditions in the conjecture are easily seen to hold.
\end{proof}

\subsection*{Organization of the paper}

The next section is devoted to proving Theorem \ref{Tmain}. In the third
and final section, we recast Theorem \ref{Tmain} in terms of zeta values
of linear polynomials. This yields a family of translation-invariant
polynomials over each finite ring. We then classify the
translation-invariant polynomials over all commutative rings featuring in
Theorem \ref{Tmain}, and in fact over any (possibly infinite) ring in a
larger class.

\section{Proof of the main result}

In proving Theorem \ref{Tmain}, the first reduction \eqref{Eproduct} to
rings of prime power characteristic is obtained via the following
straightforward computation, which is also used below. In it, we consider
the more general case of a direct product of rings $R = \times_{i=1}^l
R_i$. Writing an element $r \in R$ uniquely as $\pi_1(r) + \cdots +
\pi_l(r)$, with $\pi_i(r) \in R_i$, we obtain:
\begin{equation}\label{Etimes}
\zeta_R(-k) = \sum_{r \in R} (\pi_1(r) + \cdots + \pi_l(r))^k =
\sum_{r \in R} \left( \sum_{i=1}^l \pi_i(r)^k \right) = \sum_{i=1}^l
\frac{|R|}{|R_i|} \zeta_{R_i}(-k).
\end{equation}

Now Theorem \ref{Tmain} is proved in a series of steps. The first few
steps isolate the rings for which there exist non-vanishing power sums,
and for ease of exposition we present these steps as individual
propositions.
We begin by addressing the case of finite rings with odd prime power
characteristic.

\begin{prop}\label{Podd}
Suppose $R$ is a finite ring of characteristic $p^m$ for some $m
\geqslant 1$, with $p$ odd as above. Also suppose $R$ is neither a cyclic
group (as an additive group) nor a field. Then all power sums over $R$
vanish.
\end{prop}

The following technical lemma will be useful in the proof (and beyond).

\begin{lemma}\label{L1}
Suppose $R$ has prime characteristic $p>2$, and a nonzero two-sided ideal
$\I$ with $\I^2 = 0$. Then $\zeta_R(-k) = 0\ \forall k > 0$. The same is
true if $p=2$ and $\dim_{\F_2} (\I) \geqslant 2$.
\end{lemma}

In proving Proposition \ref{Podd}, at least when the characteristic is
$p$, the idea is to apply the lemma for $\I$ a suitable power of the
Jacobson radical of $R$.

\begin{proof}
Choose and fix coset representatives $\{ c \in \mathcal{C} \}$ of the
subspace $\I \subset R$. Writing every element $r \in R$ uniquely as $c +
y$, for $c \in \mathcal{C}$ and $y \in \I$, and rearranging sums, we have
\begin{align*}
\sum_{r \in R} r^k = \sum_{c \in \mathcal{C}} \sum_{y \in \I} (c+y)^k
= &\ \sum_{c \in \mathcal{C}} |\I| c^k + \sum_{c \in \mathcal{C}} \sum_{y
\in \I} \sum_{j=0}^{k-1} c^j y c^{k-1-j}\\
= &\ \sum_{c \in \mathcal{C}} \sum_{j=0}^{k-1} c^j \left( \sum_{y
\in \I} y \right) c^{k-1-j},
\end{align*}

\noindent and the innermost sum vanishes if $p$ is odd, or if $p=2$ and
$\dim_{\F_2} (\I) \geqslant 2$.
\end{proof}

\begin{proof}[Proof of Proposition \ref{Podd}]
The proof is by induction on $m$; we will consider the $m=1,2$ cases
separately from higher $m$. Also assume without loss of generality that
$R$ does not have characteristic $p^{m-1}$.

First suppose $m=1$. If $R$ has nonzero Jacobson radical $J(R)$, say of
nilpotence degree $d>1$, then we are done by Lemma \ref{L1} for $\I :=
J(R)^{d-1}$. Otherwise $J(R) = 0$, whence $R$ is isomorphic to a direct
product of matrix rings:
\[
R = \times_{j=1}^l M_{d_j \times d_j} ( D_j ),
\]

\noindent for finite division rings $D_j$ of characteristic $p$. By
Wedderburn's Little Theorem, each $D_j$ is a field $\F_{p^{m_j}}$, say.
Now use Equation \eqref{Etimes} to show every power sum is zero if there
is more than one factor. Finally, if $R = M_{d \times d}(\F_{p^m})$ with
$d>1$, then we are done by Theorem \ref{Tbcl}.

Now suppose $R$ has characteristic $p^m$, with $m>1$, and the result
holds for $m-1$.
Without loss of generality, we may assume $p^{m-1} R \neq 0$. Define $\I
:= p^{m-1} R$; then $\I^2 = 0$ since $m>1$. Since $p \I = 0$, note as in
Lemma \ref{L1} that $\sum_{y \in \I} y = 0$. Also choose a set
$\mathcal{C}$ of coset representatives of the subgroup $\I$ in $R$; then
as in the proof of Lemma \ref{L1},
\begin{equation}\label{Ecompute}
\sum_{r \in R} r^k = \sum_{c \in \mathcal{C}} \sum_{y \in \I} (c+y)^k
= |\I| \sum_{c \in \mathcal{C}} c^k + \sum_{c \in \mathcal{C}}
\sum_{j=0}^{k-1} c^j \left( \sum_{y \in \I} y \right) c^{k-1-j} =
|\I| \sum_{c \in \mathcal{C}} c^k.
\end{equation}

\noindent For future use, note the first two equalities also hold in
rings of characteristic $2^m$ for $m>1$.

There are now two cases. First if $m=2$, then $R / (p)$ is a module over
$\F_p$. By the above analysis in the $m=1$ case, $R / (p)$ can have
nonzero power sums only when $R / (p) = \F_q$. If $q=p$ then $R / (p)$ is
cyclic, whence so is $R$ by Nakayama's Lemma for $\Z / p^2\Z$, but this
is impossible by assumption. Hence $p^2 | q$. Now $\I = (p)$ is a module
over $R / (p) = \F_q$, so $p^2$ divides $|\I|$, and therefore the power
sum in \eqref{Ecompute} vanishes as desired.

Finally, if $m>2$, then it suffices to show, via (lifting using) the
quotient map $: R \twoheadrightarrow R / (p^{m-1})$, that the power sum
$\sum_{c \in R / (p^{m-1})} c^k$ vanishes. But since $R / (p^{m-1})$ has
characteristic $p^{m-1}$, by the induction hypothesis it can only have
nonzero power sums if it is cyclic. In this case, Nakayama's Lemma
implies $R$ is also cyclic, which is impossible.
\end{proof}

The next case to consider is if $R$ has characteristic $p=2$. In what
follows, define a ring to be \textit{indecomposable} if it is not the
direct product of two factor rings.

\begin{prop}\label{P2}
Suppose $R$ is a finite ring of characteristic $2$. Then $R$ has nonzero
power sums only if $R$ is one of $\F_{2^m}$, $M_{2 \times 2}(\F_2)$,
$\F_2[x] / (x^2)$, or the upper triangular matrices $T_{2 \times
2}(\F_2)$.
\end{prop}

\begin{proof}
At the outset we note that $R$ is indecomposable, for if $R = R_1 \times
R_2$, both of even order, then $\zeta_R(-k) = 0$ by \eqref{Etimes}.

First suppose $J(R) = 0$. Then $R = \times_{j=1}^l M_{d_j \times
d_j}(\F_{2^{m_j}})$, as in the proof of Proposition \ref{Podd}. Again by
Equation \eqref{Etimes}, every power sum is zero if there is more than
one factor. Now by Theorem \ref{Tbcl}, the only cases when $R$ has
nonzero power sums is if $R$ is a field or $M_{2 \times 2}(\F_2)$.

Next, suppose $J(R) \neq 0$. We propose an argument similar to Lemma
\ref{L1}. First suppose $\dim_{\F_2} J(R) \geqslant 2$. If $J(R)$ has
nilpotence degree $d \geqslant 2$, define $\I := J(R)^{d-1}$; then either
$\dim_{\F_2} \I \geqslant 2$, in which case we are done by Lemma
\ref{L1}, or $\I$ is one-dimensional and $\I \cdot J(R) = J(R) \cdot \I =
0$. We show in this latter case that all power sums vanish on $R$.

Let $\mathcal{C}$ (respectively $\mathcal{B}$) denote any fixed choice of
coset representatives of $J(R)$ in $R$ (respectively of $\I$ in $J(R)$).
Now compute as in \eqref{Ecompute}, using that $\I^2 = 0$ and $b \I = \I
b = 0$ for $b \in \mathcal{B}$:
\begin{align*}
\sum_{r \in R} r^k = &\ \sum_{c \in \mathcal{C}} \sum_{b \in \mathcal{B}}
\sum_{y \in \I} ((c+b)+y)^k\\
= &\ |\I| \sum_{c,b} (c+b)^k + \sum_{c,b} \sum_{j=0}^{k-1} (c+b)^j \left(
\sum_{y \in \I} y \right) (c+b)^{k-1-j}.
\end{align*}

\noindent Since $\I \neq 0$, $|\I| = 0$ in $R$. Moreover, since $(c+b)^j
y (c+b)^{k-1-j} = c^j y c^{k-1-j}$, we have
\[
\sum_{r \in R} r^k = 0 + \sum_{c \in \mathcal{C}} \sum_{j=0}^{k-1}
\sum_{b \in \mathcal{B}} c^j \left( \sum_{y \in \I} y \right) c^{k-1-j} =
\sum_{c \in \mathcal{C}} \sum_{j=0}^{k-1} \frac{|J(R)|}{|\I|} c^j \left(
\sum_{y \in \I} y \right) c^{k-1-j} = 0.
\]

This shows the result if $\dim_{\F_2} J(R) \geqslant 2$ or $J(R) = 0$.
The final case is if $\dim_{\F_2} J(R) = 1$, whence $J(R)^2 = 0$. By
Wedderburn's theorem, $R$ is a semidirect product of $J(R)$ and the
product ring $R / J(R) \cong \times_{j=1}^l M_{d_j \times
d_j}(\F_{2^{m_j}})$ for some integers $d_j, m_j > 0$, with $J(R)$ a
bimodule over $R / J(R)$. Now if $d_j > 1$ or $m_j > 1$ for any $j$, then
the simple algebra $M_{d_j \times d_j}(\F_{2^{m_j}})$ cannot have a
one-dimensional representation $J(R)$; thus it must act trivially on
$J(R)$. It follows by the indecomposability of $R$ that $d_j = m_j = 1\
\forall j$. Denote the idempotent in $\F_{2^{m_j}}$ by $1_j$. Then there
are unique $i,j$ such that $1_i J(R) = J(R) = J(R) 1_j$.

There are now two sub-cases. If $i=j$ then the product of finite fields
is a single field $\F_2$, whence $1_i$ is the unit and hence the
semidirect product $R$ is abelian. But then $R \cong \F_2[x] / (x^2)$,
which shows the result.
The other case is if $i \neq j$, in which case the algebra $R$ is
explicitly isomorphic to $T_{2 \times 2}(\F_2)$.
\end{proof}

As an intermediate step, we verify our main result in the last two cases
of Proposition \ref{P2}.

\begin{prop}\label{Pcalc2}
Theorem \ref{Tmain} holds for $R = \F_2[x] / (x^2)$ and $R = T_{2 \times
2}(\F_2)$.
\end{prop}

\begin{proof}
For $R = \F_2[x] / (x^2)$, the result follows from a straightforward
calculation. Next set $R = T_{2 \times 2}(\F_2)$, and note the powers of
individual matrices in $R$ are as follows: $\begin{pmatrix} 1 & 1\\ 0 &
1\end{pmatrix}^n = \begin{pmatrix} 1 & n\\ 0 & 1\end{pmatrix}$, the
nilpotent cone is $\F_2 x = \F_2 \begin{pmatrix} 0 & 1\\ 0 &
0\end{pmatrix}$, and the remaining five matrices are idempotent. Now a
straightforward computation shows that $\sum_{r \in R} r^k$ vanishes if
$k$ is even or $k=1$, and equals $x$ for $k>1$ odd.
\end{proof}

Our next result proves Theorem \ref{Tmain} when $R$ is not a cyclic group
and has characteristic $2^m$ for some $m>1$.

\begin{prop}\label{Peven}
Suppose $R$ has characteristic $2^m$ for some $m>1$. If $R$ is not a
cyclic group, then all power sums vanish.
\end{prop}

\begin{proof}
Without loss of generality, assume $2^{m-1} R \neq 0$.
We show the result by induction on $m \geqslant 2$.
For the base case of $m=2$, there are two sub-cases.
Set $\I := 2^{m-1} R$, and first suppose $\dim_{\F_2} \I \geqslant 2$. 
Then the result follows from the calculation in Equation
\eqref{Ecompute}, since $|\I| = 0$ in $R$ by assumption, and as explained
in the proof of Lemma \ref{L1}, $\sum_{y \in \I} y = 0$ for the
$\F_2$-vector space $\I$.

In the other sub-case $\dim_{\F_2} (2^{m-1}) = 1$, we have $\I = \{ 0,
2^{m-1} \}$, so $\I$ is central. Now Equation \eqref{Ecompute} shows:
\begin{equation}\label{Eeven}
\sum_{r \in R} r^k = 2 \sum_{c \in \mathcal{C}} c^k + 2^{m-1} k \sum_{c
\in \mathcal{C}} c^{k-1}.
\end{equation}

If $R$ has nonzero power sums, then by \eqref{Eeven}, so does the
$\F_2$-algebra $R / 2R$. It follows by Proposition \ref{P2} that $R/2R$
is the algebra of upper triangular matrices $T_{2 \times 2}(\F_2)$, or
else $\F_2[x] / (x^2), \F_{2^m}$, or $M_{2 \times 2}(\F_2)$. Since $R/2R$
has a one-dimensional simple module $2R$ of size $2$, we have $R/2R \neq
M_{2 \times 2}(\F_2), \F_{2^m}$ for $m>1$. If $R/2R = \F_2$, then by
Nakayama's Lemma $R$ is also cyclic, which is impossible.

Thus, $R/2R \cong T_{2 \times 2}(\F_2)$ or $\F_2[x]/(x^2)$. In both
cases, recall by Proposition \ref{Pcalc2} that every power sum is
nilpotent, say $\sum_{c \in R/2R} c^k =: s_k$, with $s_k^2 = 0$. Choose
any lift $\overline{s}_k \in R$ of $s_k$; then \eqref{Eeven} implies:
\[
\sum_{r \in R} r^k = 2 \overline{s}_k + 2k \overline{s}_{k-1}.
\]

\noindent We now claim that if $s^2 = 0$ in $R/2R$, then $2 \overline{s}
= 0$ for any lift $\overline{s}$; note this claim proves the $m=2$ case.
Indeed, if $2 \overline{s} \in 2R \setminus \{ 0 \}$ then $2 \overline{s}
= 2$, so $2 \overline{s}^2 = 2$. But $\overline{s}^2 \in 2R$, so we get
$2 = 2 \overline{s}^2 = 0$, which is impossible.

Having proved the result for $m=2$, the inductive step uses Equation
\eqref{Ecompute}, by adapting for $p=2$ the proof of Proposition
\ref{Podd} -- specifically, the $m>2$ case.
\end{proof}

Combining all of the above analysis, we now show the main theorem above.

\begin{proof}[Proof of Theorem \ref{Tmain}]
The first assertion in \eqref{Eproduct} was shown above in
\eqref{Etimes}. Thus, assume henceforth that $R$ has prime power
characteristic. Now the result is shown
for $R = M_{2 \times 2}(\F_2)$ via Theorem \ref{Tbcl}, and
for $R = \F_2[x] / (x^2)$ as well as $T_{2 \times 2}(\F_2)$ in
Proposition \ref{Pcalc2}.

Next, suppose $R$ is indecomposable and of prime power characteristic. By
the above analysis in Propositions \ref{Podd}, \ref{P2}, \ref{Pcalc2},
and \ref{Peven}, Theorem \ref{Tmain} is shown for all such $R$ except for
$R = \Z/p^m\Z$, where $p \geqslant 2$ is prime and $m \geqslant 2$ (so
that $R$ is not a finite field). Thus we assume henceforth that $R = \Z /
p^m\Z$ with $m \geqslant 2$, and compute $\zeta_R(-k)$.

First assume $p$ is odd or $k$ is even. We claim that for all $m \in
\mathbb{N}$,
\begin{equation}\label{Ecyclic}
\zeta_{\Z / p^{m+1} \Z}(-k) = p\; \zeta_{\Z / p^m \Z}(-k) \mod p^{m+1}.
\end{equation}

\noindent Indeed, compute for an indeterminate $x$:
\[
x^k + (x+p^m)^k + \cdots + (x+(p-1)p^m)^k \equiv p x^k + \binom{k}{1}
x^{k-1} p^m \sum_{j=0}^{p-1} j \mod p^{2m},
\]

\noindent whence the same equality holds modulo $p^{m+1}$. Now if either
$p$ is odd or $p,k$ are even, then the second term on the right vanishes.
It follows that
\[
\zeta_{\Z / p^{m+1} \Z}(-k) = \sum_{i=0}^{p^m-1} \sum_{j=0}^{p-1} (i +
jp^m)^k \equiv p \sum_{i=0}^{p^m-1} i^k \mod p^{m+1},
\]

\noindent upon applying the above analysis to $x \leadsto i$. This shows
\eqref{Ecyclic}, and hence, the result.

It remains to consider the case when $p=2 \leqslant m$ and $k$ is odd. If
$k=1$ then the result is straightforward. We now claim that if $k>1$ is
odd, then
\[
\zeta_{\Z / 2^m \Z}(-k) = 0, \qquad \forall m \geqslant 2.
\]

The proof is by induction on $m$. The $m=2$ case is immediate. Next, if $m
\geqslant 2$,
\[
\zeta_{\Z / 2^{m+1} \Z}(-k) = \sum_{i=0}^{2^m-1} [i^k + (2^m+i)^k]
= \sum_{i=0}^{2^m-1} [2 i^k + 2^m k i^{k-1}].
\]

\noindent Modulo $2^{m+1}$, the coefficient $2^m k$ in the summand can be
replaced by $2^m$, whence
\begin{equation}\label{Etemp}
\zeta_{\Z / 2^{m+1} \Z}(-k) \equiv 2 \zeta_{\Z / 2^m \Z}(-k) + 2^m
\zeta_{\Z / 2^m \Z}(-(k-1)) \mod 2^{m+1}.
\end{equation}

\noindent By the above analysis, the last term on the right equals
\[
2^m \cdot 2^{m-1} \zeta_{\F_2}(-(k-1)) \mod 2^{m+1},
\]
and this vanishes since $m \geqslant 2$. The first term on the right-hand
side of \eqref{Etemp} also vanishes modulo $2^{m+1}$, given the induction
hypothesis.

The above arguments show Theorem \ref{Tmain} for every indecomposable
finite unital ring $R$.
Now suppose $R$ is a general finite unital ring of prime power
characteristic, say $p^m$. Write $R = \times_j R_j$ as a product of
indecomposable factors; then by the above analysis, the prime $p$ kills
$\zeta_{R_j}(-k)$ for all $j$.
Thus, if there are at least two factors $R_j$, then $\zeta_R(-k) = 0$ by
\eqref{Etimes}. This concludes the proof of the theorem.
\end{proof}

\section{Translation-invariant polynomials over prime power
characteristic rings}

In this section we work over polynomial rings $R[T]$, where $T$ is an
indeterminate and $R$ an arbitrary finite ring as above (later in the
section, we relax the finiteness assumption). Our goal is to prove two
results. The first result translates the problem of computing zeta values
$\zeta_R(-k)$ into computing
\[
\calp_k^R(T) := \sum_{r \in R} (T+r)^k.
\]

If $R$ is a finite field, then this goes back to (and is subsumed by) the
more general work of Carlitz; we include \cite{Ca1,Ca2} as well as the
survey article \cite{Ca3}. The setting of Carlitz involves the above sum
being replaced by the sum over all monic polynomials of specified degree
$d$, i.e.,
\[
\calp_k^{\F_q,d}(T) := \sum_{f \in \F_q[T],\ \deg(f) < d} (T^d + f(T))^k,
\qquad k \in \Z.
\]

\noindent Summing over $d \geqslant 0$ yields the zeta function of
$\F_q[T]$, which has important connections to (rank one) Drinfeld modules
and to $q$-expansions of Eisenstein series. The congruence properties of
these sums are also linked to class number formulas, leading to
Kummer-type criteria for abelian extensions of $\F_q(T)$; see
e.g.~\cite{Gos2}. For completeness, we point the reader to \cite{Ge,Gos1}
and the recent survey article \cite{Th} for more on these connections,
related topics, and references.

Returning to our setting over a general finite ring, we begin by
reformulating Theorem \ref{Tmain} into a polynomial \textit{avatar}.
Namely: the following result computes the aforementioned polynomials
$\calp_k^R(T)$ for every $k>0$ and finite ring $R$.

\begin{theorem}\label{Tpoly}
(Notation as in Theorem \ref{Tmain}.)
Suppose $R = \times_{i=1}^l R_i$ is an arbitrary finite unital ring. Then
\begin{equation}
\calp^R_k(T) = \sum_{i=1}^l \frac{|R|}{|R_i|} \calp^{R_i}_k(T),
\end{equation}

\noindent where $k \in \mathbb{N}$, and the $i$th summand lives in
$R_i[T]$.

Now consider the problem for each factor $R = R_i$, with characteristic
equal to a power of a prime $p>0$. Letting $x := \begin{pmatrix} 0 & 1\\
0 & 0 \end{pmatrix}$, we have:
\begin{equation}
\calp^R_k(T) = \begin{cases}
\calp_k^{\F_q}(T), & \text{if } R = \F_q, \text{ where } q = p^e;\\
p^{m-1} \calp_k^{\F_p}(T), & \text{if } R = \Z / p^m \Z, \text{ and } p>2
\text{ or } p = 2|k;\\
2^{m-1} \calp_k^{\F_2}(T) + 2^{m-1} \calp_{k-1}^{\F_2}(T), &
\text{if } R = \Z / 2^m \Z, \text{ and } 2 \nmid k;\\
x \calp_{k-1}^{\F_2}(T) = x \sum_{1<j\ {\rm odd}} \binom{k}{j} T^{k-j},
& \text{if } R = \F_2[x] / (x^2)
\text{ or } T_{2 \times 2}(\F_2), \text{ and } 2 \nmid k;\\
{\rm Id}_{2 \times 2} \cdot \sum_{1<j \equiv 0, \pm 1 \hspace*{-2mm} \mod
6} \binom{k}{j} T^{k-j}, & \text{if } R = M_{2 \times 2}(\F_2);\\
0, & \text{otherwise}.
\end{cases}
\end{equation}

\noindent Here, for any finite field we have:
\begin{equation}\label{Ewaring}
\calp_k^{\F_q}(T) = - \sum_{\gamma = \lfloor k/q \rfloor + 1}^{\lfloor
k/(q-1) \rfloor} \binom{\gamma - 1}{k - (q-1)\gamma}
(T^q-T)^{k-(q-1)\gamma} \in \F_p[T^q-T].
\end{equation}
\end{theorem}

The proof of the final identity \eqref{Ewaring} uses a preliminary
result.

\begin{lemma}\label{P10}
Given $1 \leqslant k \leqslant q$, define
\[
\fk := \{ S \subset \F_q : |S| = k \}, \quad
\Sigma_k(T) := \sum_{S \in \fk} \prod_{r \in S} (T+r), \quad
\Sigma_{-k}(T) := \sum_{S \in \fk} \prod_{r \in S} (T+r)^{-1}.
\]

\noindent Then for $1 \leqslant |k| \leqslant q$,
\begin{equation}
\Sigma_k(T) =
\begin{cases}
    -1, 	&\text{if $k=q-1$;}\\
    T^q-T,	&\text{if $k=q$;}\\
    -1/(T^q-T) = S_{-1}(T),	&\text{if $k=-1$;}\\
    1/(T^q-T),	&\text{if $k=-q$;}\\
    0,		&\text{otherwise.}
\end{cases}
\end{equation}
\end{lemma}

\begin{proof}
Since each $\Sigma_k(T)$ is also translation-invariant, by standard
Galois theory it is a polynomial in $T^q-T$, with $T$-degree at most $k$.
Thus for $0 \leqslant k<q$, $\Sigma_k(T)$ is a constant, hence equals
$\Sigma_k(0)$. Now compute using Equation \eqref{Efactor}:
\[
T^q - T = \prod_{r \in \F_q} (T-r) = T^q - T^{q-1} \Sigma_1(0) + T^{q-2}
\Sigma_2(0) - \dots + (-1)^q \Sigma_q(0).
\]

\noindent More precisely, replace $n$ by $q$ and apply the homomorphism
$: \F_q[\{ x_r : r \in \F_q \}] \to \F_q$, sending $x_r \mapsto r$, to
Equation \eqref{Efactor}.
Thus, $\Sigma_1(0) = \cdots = \Sigma_{q-2}(0) = 0$, and
\[
(-1)^{q-1} \Sigma_{q-1}(0) = -1.
\]

\noindent If $q$ is even, then $p=2$, and $(-1)^{q-1} = \pm 1 = 1$ in
$\F_q$. Otherwise, $q$ is odd, and $(-1)^{q-1} = 1$. In either case, we
get: $\Sigma_{q-1}(0) = -1$.

Next if $k=q$, then $\fk = \{ \F_q \}$, so $\Sigma_q(T) = \prod_{r \in
\F_q} (T-r) = T^q-T$. This also shows that $\Sigma_{-q}(T) = 1 / (T^q -
T)$.
Finally, if $0 < k < q$, then taking common denominators yields:
$\Sigma_{-k}(T) = \Sigma_{q-k}(T) / \Sigma_q(T)$. This concludes the
proof.
\end{proof}

Now we can prove the above theorem.

\begin{proof}[Proof of Theorem \ref{Tpoly}]
We prove the result in reverse order of the cases presented, ending by
proving the case of $R = \F_q$ in \eqref{Ewaring}. Most of the result
follows from Theorem \ref{Tmain}, via the identity:
\begin{equation}\label{Esum}
\calp_k^R(T) = \sum_{r \in R} (T+r)^k = |R| T^k + \sum_{j=1}^k
\binom{k}{j} \zeta_R(-j) T^{k-j}.
\end{equation}

\noindent In particular, if $\sum_{r \in R} r^k = 0$ for all $k>0$, then
$\calp_k^R(T) = |R| T^k = 0$. Similarly, the result for $R = M_{2 \times
2}(\F_2)$ follows from Theorem \ref{Tbcl} via \eqref{Esum}.

Continuing \textit{up} the list of cases in Theorem \ref{Tpoly}, we have
$R = \F_2[x] / (x^2)$ or $T_{2 \times 2}(\F_2)$. If $k$ is even, then from
\eqref{Esum} we obtain:
\[
\calp_k^R(T) = \sum_{j>0, \ j \text{ odd}} \binom{k}{j} \zeta_R(-j)
T^{k-j} = x \sum_{j > 0, \ j \text{ odd}} \binom{k}{j} T^{k-j},
\]
using Theorem \ref{Tmain}. Thus it suffices to show that $\binom{k}{j} =
0$ for $k$ even and $j$ odd. This follows from Lucas' theorem \cite{Lu};
we also provide a direct proof. Consider the coefficient of $y^j$ in
expanding $(1+y)^k$ over $\Z/2\Z$. If $k=2l$, then we compute using the
Frobenius:
\[
(1+y)^k = ((1+y)^2)^l = (1+y^2)^l,
\]

\noindent from which it follows that the coefficient of $y^j$ vanishes
modulo $2$, as desired.

Otherwise if $k$ is odd, then from Pascal's triangle and the previous
paragraph, we have:
\[
\binom{k}{j} \equiv \binom{k-1}{j-1} + \binom{k-1}{j} \equiv
\binom{k-1}{j-1} \mod 2,
\]
whenever $1 \leqslant j \leqslant k$ and $j,k$ are odd. Using
\eqref{Esum} and Theorem \ref{Tmain}, as well as the previous paragraph,
it follows that
\begin{align*}
\calp_k^R(T) = &\ x \sum_{j > 1, \ j \text{ odd}} \binom{k}{j} T^{k-j}
= x \sum_{j > 0, \ j \text{ even}} \binom{k-1}{j} T^{k-j}\\
= &\ x \sum_{j > 0} \binom{k-1}{j} T^{k-j} = x (T^{k-1} + (1+T)^{k-1})
= x \calp_{k-1}^{\F_2}(T).
\end{align*}
This proves the relevant case of Theorem \ref{Tpoly}; note that if $k=1$
then we obtain $x \calp_0^{\F_2}(T) = 0$.

Next suppose $R = \Z / p^m \Z$, with $p \geqslant 2$ prime and $m
\geqslant 2$.
First if $p$ is odd or $k$ is even or $k=1$, then we are done by
\eqref{Esum} and Theorem \ref{Tmain}.
Otherwise suppose $p=2 < k$, with $k$ odd. We first show the ``induction
step'', relating $\calp_k^{\Z / 2^{e+1} \Z}$ to $\calp_k^{\Z / 2^e \Z}$:
\[
\calp_k^{\Z / 2^{e+1} \Z}(T) = \sum_{i=0}^{2^{e+1}-1} (T+i)^k =
\sum_{i=0}^{2^e-1} [(T+i)^k + (T+2^e+i)^k] = \sum_{i=0}^{2^e-1} [2
(T+i)^k + 2^e k (T+i)^{k-1}].
\]

\noindent Modulo $2^{e+1}$, the coefficient $2^e k$ in the summand can be
replaced by $2^e$, whence
\[
\calp_k^{\Z / 2^{e+1} \Z}(T) \equiv 2 \calp_k^{\Z / 2^e \Z}(T) + 2^e
\calp_{k-1}^{\Z / 2^e \Z}(T) \mod 2^{e+1}.
\]

\noindent By the previous case, the last term equals
\[
2^e \calp_{k-1}^{\Z / 2^e \Z}(T) \equiv 2^e \cdot 2^{e-1}
\calp_{k-1}^{\F_2}(T) \mod 2^{e+1},
\]

\noindent and this vanishes if $e \geqslant 2 = p$. Repeating this
computation inductively,
\begin{align*}
\calp_k^{\Z / 2^m \Z}(T)
\equiv &\ 2 \calp_k^{\Z / 2^{m-1} \Z}(T) + 2^{m-1}
\calp_{k-1}^{\Z/2^{m-1}\Z}(T)\\
\equiv &\ 4 \calp_k^{\Z / 2^{m-2} \Z}(T) + 2^{m-1} \calp_{k-1}^{\Z/
2^{m-2} \Z}(T) + 0\\
\equiv &\ \cdots\\
\equiv &\ 2^{m-2} \calp_k^{\Z / 4 \Z}(T) + 2^{m-1} \calp_{k-1}^{\Z / 4
\Z}(T) + 0\\
\equiv &\ 2^{m-1} \calp_k^{\Z / 2 \Z}(T) + 2^{m-1} \calp_{k-1}^{\Z / 2
\Z}(T) + 0 \mod 2^m.
\end{align*}

\noindent This shows the result when $p=2$ and $k$ is odd, and hence for
all rings $R = \Z / p^m \Z$.

The final step is to show \eqref{Ewaring}. For this we recall the theory
of symmetric polynomials over a commutative unital ring $A$. Note that
the elementary symmetric polynomials
\[
\sigma_k := \sum_{1 \leqslant i_1 < i_2 < \dots < i_k \leqslant n}
x_{i_1} x_{i_2} \cdots x_{i_k} \qquad (1 \leqslant k \leqslant n),
\]
which appear when expanding a linear factorization of a monic polynomial:
\begin{equation}\label{Efactor}
\prod_{i=1}^n (T - x_i) = T^n - \sigma_1 T^{n-1} + \sigma_2 T^{n-2} +
\dots + (-1)^n \sigma_n,
\end{equation}

\noindent can be used to write the power-sum polynomials
\[
p_k(x_1, \dots, x_n) := \sum_{j=1}^n x_j^k, \qquad k \geqslant 0
\]
via the following well-known formula:

\begin{theorem}[Waring's Formula; see e.g. \cite{Gou}]\label{T9}
Given $R[x_1, \dots, x_n]$, for all $k>0$,
\[
p_k = \sum (-1)^{i_2 + i_4 + i_6 + \cdots} \frac{(i_1 + i_2 + \cdots +
i_n - 1)! k}{i_1! i_2! \cdots i_n!} \sigma_1^{i_1} \sigma_2^{i_2} \cdots
\sigma_n^{i_n} \in \Z[\sigma_1, \sigma_2, \dots, \sigma_n],
\]

\noindent where the sum is over all $(i_1, \dots, i_n) \in \Z^n$ such
that all $i_j \geqslant 0$ and $\sum_j i_j j = k$.
\end{theorem}

Now we conclude the proof of the theorem by showing \eqref{Ewaring}. We
first claim:
\begin{equation}\label{Eintermediate}
\calp_k^{\F_q}(T) = \sum_{0 \leqslant \alpha, \beta \in \Z : (q-1)\alpha
+ q\beta = k} \frac{(\alpha + \beta - 1)! k}{\alpha! \beta!} (T^q -
T)^\beta \in \F_p[T].
\end{equation}

Apply Waring's Formula (from Theorem \ref{T9}) and Lemma \ref{P10}. Set
$R = \F_q$ and consider the $\F_p$-algebra homomorphism $: \F_p[\{ x_r :
r \in \F_q \}] \to \F_q[T]$, sending $x_r \mapsto T+r$. This sends
$\sigma_k \mapsto \Sigma_k(T)$ for all $1 \leqslant k \leqslant q$, and
$p_k \mapsto \calp_k^{\F_q}(T)\ \forall k>0$. Hence $\calp_k^{\F_q}(T)$
is computable using Waring's Formula. Notice that every summand
containing a positive power of (the image of) $\sigma_1, \sigma_2,
\dots$, or $\sigma_{q-2}$ vanishes by Lemma \ref{P10}. Now changing
indices from $i_{q-1}, i_q$ to $\alpha, \beta$ respectively,
\[
\calp_k^{\F_q}(T) = \sum_{0\leqslant \alpha, \beta \in \Z : \alpha(q-1) +
\beta q = k} (-1)^{\epsilon} \frac{(\alpha + \beta - 1)!k}{\alpha!
\beta!} \Sigma_{q-1}(T)^\alpha \Sigma_q(T)^\beta,
\]

\noindent with two cases. If $q$ is odd, then $\epsilon = \alpha$, and
the summand equals
\[
\frac{(\alpha + \beta - 1)!k}{\alpha! \beta!} \cdot
(-1)^\alpha \Sigma_{q-1}(T)^\alpha = 
\frac{(\alpha + \beta - 1)!k}{\alpha! \beta!},
\]

\noindent which proves \eqref{Eintermediate}.
If instead $q$ is even, then $\epsilon = \beta$, whence
\eqref{Eintermediate} follows similarly.

Having shown \eqref{Eintermediate}, we now derive \eqref{Ewaring} from
it. If $\alpha = 0$ then the corresponding coefficient of $(T^q-T)^\beta$
is $\frac{(\beta-1)! (q \beta)}{\beta!} = q = 0$ in $\F_q$. Thus we may
assume $\alpha > 0$. Now change variables from $\beta$ to $\gamma =
\alpha + \beta$. Then Equation \eqref{Eintermediate} yields:
\[
\calp_k^{\F_q}(T) = \sum_{0 < \alpha \leqslant \gamma\ :\ q \gamma -
\alpha = k} \frac{(\gamma - 1)! k}{\alpha! (\gamma - \alpha)!} (T^q -
T)^{\gamma - \alpha}.
\]

\noindent Since $\alpha > 0$, and $\beta = \gamma - \alpha = k -
(q-1)\gamma$, it follows that
\[
\frac{(\gamma - 1)! k}{\alpha! (\gamma - \alpha)!} = 
\frac{(\gamma - 1)! q \gamma}{\alpha! (\gamma - \alpha)!} -
\frac{(\gamma - 1)! \alpha}{\alpha! (\gamma - \alpha)!}
= -\binom{\gamma-1}{k - (q-1)\gamma}.
\]


\noindent Now the result \eqref{Ewaring} follows by putting together the
above arguments. Finally, note that the coefficients are in $\Z / p\Z$,
because Waring's Formula has integer coefficients.
\end{proof}

\begin{remark}
Using Lemma \ref{P10} and Waring's formula, one shows similarly that
\begin{align}
\begin{aligned}
\calp_{-k}^{\F_q}(T) = &\ \sum_{0 \leqslant \alpha, \beta \in \Z\ :\
\alpha + q\beta = k} \frac{(\alpha + \beta - 1)! k}{\alpha! \beta!}
\frac{(-1)^\alpha}{(T^q - T)^{\alpha+\beta}}\\
= &\ \sum_{\beta=0}^{\lceil k/q \rceil - 1} \binom{k-(q-1)\beta-1}{\beta}
\frac{(-1)^{k-q\beta}}{(T^q-T)^{k-(q-1)\beta}},
\end{aligned}
\end{align}

\noindent where $\alpha = k - q\beta$, and $k>0$. Note, the
translation-invariant polynomial (in the variables $1/(T+r)$ for $r \in
\F_q$) lies in $\F_p[(T^q-T)^{-1}] = \F_p[\Sigma_{-q}(T)]$.
\end{remark}

Notice that the polynomials $\calp_k^R(T)$ are translation-invariant:
$\calp_k^R(T+r) = \calp_k^R(T)$ for all $r \in R$. The final result of
this paper is in this spirit, and classifies the translation-invariant
polynomials over the three finite commutative rings that feature in
Theorem \ref{Tmain}: $\F_q$, $\Z/p^m\Z$, and $\F_2[x]/(x^2)$. More
strongly: we combine the three rings into a common framework using Witt
vectors \cite{Wi}, and classify the translation-invariant polynomials
over all such rings, finite or not.

\begin{theorem}\label{Twitt}
Fix integers $e,m > 0$ and a prime power $q = p^e$ with $p>0$.
Denote by $\W_m(\F_q)$ the corresponding Witt ring of vectors of length
$m$. Now suppose $R = \m \oplus \W_m(\F_q)$ is a commutative unital local
ring, where $\m$ is a free $\W_m(\F_q)$-module, as well as an ideal of
nilpotence class $k$ for some $1 \leqslant k \leqslant p$. Then the set
of translation-invariant polynomials in $R[T]$ is given by
\begin{equation}\label{Ewitt}
\bigoplus_{n \geqslant 0} R \cdot (T^q-T)^{n p^m} \oplus
\bigoplus_{i=0}^{m-1} \bigoplus_{p \nmid n} p^{m-1-i} \Ann_R(p^{m-1} \m)
\cdot (T^q-T)^{n p^i}.
\end{equation}
\end{theorem}


\noindent Observe that the formula in \eqref{Ewitt} does not depend on
the nilpotence class $k$ as long as $k \leqslant p$.

Theorem \ref{Twitt} fits into a broader setting, studied in \cite{SV}:
consider a commutative unital ring $R_m$ that is flat over $\Z / p^m \Z$
for some $m>0$ and a fixed prime integer $p>0$, and suppose
$A_m$ is a flat $R_m$-algebra. (In our setting, $R_m = \W_m(\F_q)$, or
more generally $\W_m(\F_q)[x] / (x^k)$ for $k>0$, and $A_m = R_m[T]$.)
The notion of translation-invariance is now replaced by the more general
situation of invariance under a set $S$ of $R_m$-algebra automorphisms of
$A_m$. Then the treatment in \cite{SV} suggests the following recipe to
produce $S$-invariants in $A_m$.

\begin{lemma}\label{Llift}
Under the above setup, define for $1 \leqslant j \leqslant m$ the algebra
\[
R_j := R_m \otimes_{\Z / p^m \Z} \Z / p^j \Z,
\]

\noindent and $A_j$ to be the corresponding algebra to $A_m$ but defined
over $R_j$. Now if $\overline{S}$ denotes the group of automorphisms
induced on the $R_1$-algebra $A_m / p A_m = A_1$, and $a_1 \in
A_1^{\overline{S}}$ is fixed by $\overline{S}$, then the elements
\[
p^{m-1-i} \widetilde{a_1}^{p^i} \in A_m, \qquad 0 \leqslant i < m
\]

\noindent are all fixed by $S$, where $\widetilde{a_1} \in A_m$ denotes
any lift of $a_1 \in A_1$.
\end{lemma}


Lemma \ref{Llift} provides a recipe to produce invariant elements in
$A_m$ using $\overline{S}$-invariants of $A_1$. A more challenging
question in this general setting and in specific examples involves
understanding if this recipe generates \textit{all} $S$-invariants in
$A_m$. Theorem \ref{Twitt} shows that this is indeed the case for the
family $R_m = \W_m(\F_q)$, including the finite rings $\F_q$ and $\Z /
p^m \Z$ that feature in Theorem \ref{Tmain} above.\medskip

We now turn to Theorem \ref{Twitt}. Before showing the result, we present
an alternate formulation in terms of the coefficients $c_n \in R$ for
which $c_n (T^q-T)^n$ is translation-invariant. First define for each
$n>0$ the integer $v_p(n)$ to be the $p$-adic valuation, and also,
\[
n_\downarrow := \lfloor p^{m-1-v_p(n)} \rfloor p^{v_p(n)} =
\begin{cases}
0, \qquad & \text{ if } p^m \mid n,\\
p^{m-1}, & \text{ otherwise}.
\end{cases}
\]

\noindent Now the set of translation-invariant polynomials is (claimed to
be) precisely the space
\begin{equation}
R \oplus \bigoplus_{n>0} \lceil p^{m-1-v_p(n)} \rceil \cdot
\Ann_R(n_\downarrow \cdot \m) \cdot (T^q-T)^n.
\end{equation}

\begin{proof}[Proof of Theorem \ref{Twitt}]
We begin with the following observation:\medskip

\textit{Suppose $R$ is a commutative unital ring with prime ideal
generated by a prime integer $p>0$. Now suppose the following equation
holds in $R$: $f_1 = f_0 + p^j X$, with $j>0$. Fix $n>0$ such that $p^2
\nmid n$. Then there exists $X' \in R$ such that
\begin{equation}\label{Ecalculation}
f_1^n = \begin{cases}
f_0^n + p^j X', & \text{ if } p \nmid n,\\
f_0^n + p^{j+1} X', \qquad & \text{ if } p \mid n, \ p^2 \nmid n.
\end{cases}
\end{equation}
If moreover $p \nmid X$ and $p>2$, then we can choose $p \nmid
X'$.}\medskip

%
%

Now we continue with the proof, breaking it up into steps for ease of
exposition.\medskip

\noindent \textbf{Step 1:}
We first show that the claimed polynomials in \eqref{Ewitt} are indeed
translation-invariant. This is itself shown in a series of sub-steps.
First we consider the case $m=1$, i.e., $R = \W_m(\F_q) = \F_q$. Here the
polynomial $T^q-T$ is translation-invariant (by repeatedly using the
Frobenius map $\frob_q : r \mapsto r^p$), whence the assertion follows
for $R = \F_q$. Now since $(T^q-T)^n$ is translation-invariant over
$\F_q$ for $n \geqslant 0$, Lemma \ref{Llift} shows the polynomials in
\eqref{Ewitt} are translation-invariant over $R_m = \W_m(\F_q)$, where
$\m = 0$ and $\mathfrak{I} = R_m$.

We next show the assertion for $R = \m \oplus \W_m(\F_q)$. Begin by
defining and studying
\begin{equation}
f_{i,n}(T) := (T^q-T)^{n p^i}, \qquad i,n \geqslant 0.
\end{equation}

\noindent Repeatedly using $\frob_q$ shows $(T+r)^q \equiv T^q + r^q \mod
(p)$ in $R[T]$, for any $r \in R$. Thus,
\begin{equation}\label{Efirst}
f_{0,1}(T+r) - f_{0,1}(T) \equiv (T^q + r^q) - (T+r) - (T^q-T) \equiv
r^q-r \mod p T.
\end{equation}

Now to show the assertion, it suffices to show the invariance of
\eqref{Ewitt} under translation by any $r \in \m$. Start with
\eqref{Efirst}, noting by assumption on $\m$ that $r^q = 0$. Thus
$f_{0,1}(T+r) = (T+r)^q - (T+r) \equiv T^q - T - r \mod p$. Applying
\eqref{Ecalculation} for $p \nmid n$,
\[
f_{0,n}(T+r) = f_{0,1}(T+r)^n \equiv (T^q-T-r)^n \mod p,
\]

\noindent whence repeated applications of \eqref{Ecalculation} now yield:
\begin{equation}
f_{i,n}(T+r) \equiv (T^q-T-r)^{n p^i} \mod p^{i+1}.
\end{equation}

There are now two sub-cases. If $i \geqslant m$ and $n>0$, then by the
binomial theorem,
\begin{align*}
f_{m,n} (T+r) = &\ f_{m-1,n}(T+r)^p = ((T^q-T-r)^{n p^{m-1}})^p\\
= &\ (T^q-T)^{n p^m} + \sum_{l=1}^{p-1} \binom{np^m}{l} (T^q-T)^{n p^m-l}
(-r)^l\\
\in &\ f_{m,n}(T) + p^m \cdot R[T] = f_{m,n}(T),
\end{align*}

\noindent where the summation stops at most by $l=p-1$ since $r^p \in
\m^p = 0$. This calculation shows that $(T^q-T)^{n p^m}$ is
translation-invariant. Next, suppose $0 \leqslant i \leqslant m-1$, and
$r' \in \Ann_R(p^{m-1} \m)$. Given $r \in \m$, compute as above:
\begin{align*}
p^{m-1-i} r' f_{i,n} (T+r) = &\ p^{m-1-i} r' (T^q-T-r)^{n p^i}\\
= &\ p^{m-1-i} r' (T^q-T)^{n p^i} + p^{m-1-i} r'
\sum_{l=1}^{p-1} \binom{np^i}{l} (T^q-T)^{n p^i-l} (-r)^l\\
\in &\ p^{m-1-i} r' f_{i,n}(T) + r' p^{m-1-i} \cdot p^i \cdot \m[T].
\end{align*}

\noindent By choice of $r'$, it follows that $r' p^{m-1-i} f_{i,n}(T)$ is
indeed translation-invariant on $R$, as desired.\medskip

\noindent \textbf{Step 2:}
The remaining steps will prove the reverse inclusion, which is more
involved. In this step, we show that for all $m>0$ and prime powers $q$,
every translation-invariant polynomial $f(T)$ on $R = \m \oplus
\W_m(\F_q)$ is a polynomial in $T^q-T$.

To see why, first use the $\W_m(\F_q)$-freeness of $\m$ to write $f(T) =
\sum_j r_j g_j(T)$, where the $r_j$ comprise a $\W_m(\F_q)$-basis of $R$,
and $g_j \in  W_m(\F_q)[T]\ \forall j$. Then each $g_j$ is
translation-invariant on $\W_m(\F_q)$, and it suffices to show that $g_j
\in \W_m(\F_q)[T^q-T]$, for each $j$. We may replace $g_j$ by $g_j(T) -
g_j(0)$, whence $T | g_j(T)$.

Let $\omega_m : \F_q \to R = \W_m(\F_q)$ denote the Teichm\"uller
character, which restricts to an injective group morphism $: \F_q^\times
\to R^\times = R \setminus (p)$. From this it follows that $\prod_{a \in
\F_q^\times} (T - \omega_m(a)) = T^{q-1}-1$.
Moreover, the linear polynomials $\{ T-\omega_m(a) : a \in \F_q \}$ are
pairwise coprime.

Returning to the proof, since $T | g_j(T)$, it follows that $T -
\omega_m(a)$ also divides $g_j(T)$ for $a \in \F_q^\times$, whence $T^q-T
= \prod_{a \in \F_q} (T - \omega_m(a))$ divides $g_j(T)$. From this it
follows by induction on $\deg(g_j)$ that each $g_j$, whence $f$, is a
polynomial in $T^q-T$.\medskip

\noindent \textbf{Step 3:}
By the two previous steps, any $f \in R[T]$ which is
translation-invariant on $R$, is a polynomial in $T^q-T$; moreover, we
may subtract all terms of the form $c (T^q-T)^{np^m}$ as all such terms
are translation-invariant. Thus, assume henceforth that
\begin{equation}\label{Edirectsum}
f(T) \in \bigoplus_{i=0}^{m-1} \bigoplus_{p \nmid n} R \cdot
(T^q-T)^{np^i}.
\end{equation}

In this step we complete the proof for $m=1$ and all $p,q$. Notice that
it suffices to show translation-invariance by $\m$. Let $f(T) =
\sum_{j=0}^n r_j T^j$, with $n = \deg(f)$. As $p \nmid n$, use the
Frobenius repeatedly to compute for $r \in \m$:
\begin{align*}
f(T+r) - f(T) = &\ \sum_{j=0}^n r_j [ ((T+r)^q - (T+r)^j - (T^q-T)^j]\\
= &\ \sum_{j=0}^n r_j [(T^q-T-r)^j - (T^q-T)^j],
\end{align*}

\noindent so the highest degree term in $T$ comes from the leading term:
\[
- r_n \cdot n \cdot (T^q-T)^{n-1} \cdot r, \qquad r \in \m.
\]

\noindent This expression must vanish; as $n$ is invertible,
$r_n \in \Ann_R(\m)$.
Subtracting the translation-invariant polynomial $r_n (T^q-T)^n$ from
$f(T)$, we are done by induction on $\deg(f) = n$.\medskip

\noindent \textbf{Step 4:}
We next \textbf{claim} that if a translation-invariant polynomial $f$ is
as in \eqref{Edirectsum}, then $f(T)$ in fact lies in the second direct
sum in \eqref{Ewitt}. We will show the claim first for $R = \W_m(\F_p) =
\Z / p^m \Z$, then for $R = \W_m(\F_q)$, and finally (in the next step)
for $R = \m \oplus \W_m(\F_q)$, to conclude the proof of the theorem.

Begin by assuming $R = \Z / p^m \Z = \W_m(\F_p)$. Suppose $f(T) = a T^n +
\cdots$, where $n = \deg(f)$ is positive and divisible by $q=p$; and $p^m
\nmid a$. Since $f(T+1) \equiv f(T)$, it follows that $a n = 0 \mod p^m$.
Suppose $n = n' p^{i+1}$ and $a = x p^{m-1-i}$ for some $0 \leqslant i <
m$ and $n'>0$, and $0 \neq x \in \Z/p^m\Z$. Then $f(T) - x p^{m-1-i}
f_{i,n'}(T)$ is a translation-invariant polynomial of strictly smaller
degree. Therefore the claim for $\W_m(\F_p)$ follows by induction on
$\deg(f)$.

Next, suppose $R = \W_m(\F_q)$, and $f$ as in \eqref{Edirectsum} is
translation-invariant on $R$. Use the $\Z/p^m\Z$-freeness of $\W_m(\F_q)$
to write $f(T) = \sum_j r_j g_j(T)$, where the $r_j$ comprise a
$(\Z/p^m\Z)$-basis of $R$, and $g_j \in (\Z/p^m\Z)[T]\ \forall j$. Then
$g_j$ is translation-invariant on $\Z/p^m\Z$ for all $j$. Now use the
above analysis in this step for $R = \Z / p^m \Z$ to write:
\[
f(T) = \sum_{i=0}^{m-2} \sum_{j=1}^{l_i} c_{ij} p^{m-1-i} (T^p-T)^{n_{ij}
p^i} + \sum_{j=1}^{l_{m-1}} c_{m-1,j} (T^p-T)^{n_{m-1,j} p^{m-1}},
\]

\noindent where $p \nmid n_{ij}$ if $i<m-1$. We now show that $f$ is of
the desired form \eqref{Edirectsum} by induction on $m$. The base case of
$m=1$ was proved in Step 2 above; moreover if $p$ divides $c_{m-1,j}$ for
all $j$ then $p^{-1} f(T)$ makes sense and is translation-invariant in
$\W_{m-1}(\F_q)$. But then the result follows by the induction
hypothesis.

Thus, assume without loss of generality that $m>1$ and $p \nmid f(T)$.
Now consider $f(T) \mod p$, which is a polynomial over $\F_q =
\W_1(\F_q)$ that is translation-invariant. Note that
\begin{align*}
f(T) \mod p \equiv &\ \sum_{j=1}^{l_{m-1}} \overline{c}_{m-1,j}
(T^p-T)^{n_{m-1,j} p^{m-1}}\\
\equiv &\ \frob_q^{m-1}(f_1(T)), \qquad
f_1(T) := \sum_{j=1}^{l_{m-1}} \frob_q^{-(m-1)}(\overline{c}_{m-1,j})
(T^p-T)^{n_{m-1,j}}.
\end{align*}

\noindent It follows that $f_1(T)$ is translation-invariant on $\F_q$.
Now since $m=1$ for $f_1$, it follows by Step 2 that $0 \neq f_1(T) \in
\F_q[T^q-T]$, say $\displaystyle f_1(T) = \sum_j a_j (T^q-T)^j$. Then,
\[
f(T) \mod p = \frob_q^{m-1}(f_1(T)) = \sum_j a_j^{p^{m-1}} (T^q-T)^{j
p^{m-1}}.
\]

\noindent Fixing lifts $b_j \in \W_m(\F_q)$ of $a_j^{p^{m-1}} \in \F_q$
for all $j$, define
\[
h(T) := \sum_j b_j (T^q-T)^{j p^{m-1}} \in \W_m(\F_q)[T].
\]

\noindent By Step 1, $h(T)$ is translation-invariant on $\W_m(\F_q)[T]$,
and $f - h$ is a translation-invariant polynomial divisible by $p$. As
mentioned above, now divide by $p$ and work in $\W_{m-1}(\F_q)[T]$ by
induction on $m$, to conclude the proof in the case $R =
\W_m(\F_q)$.\medskip

\noindent \textbf{Step 5:}
The final step is to show the result for $R = \m \oplus \W_m(\F_q)$, with
$f(T)$ as in \eqref{Edirectsum} being translation-invariant over $R$.
Use a $\W_m(\F_q)$-basis of $R$, say $\{ r_j \}$, to write $f(T) = \sum_j
r_j g_j(T)$; then each $g_j$ is translation-invariant over $\W_m(\F_q)$.
Thus by Step 4,
\[
f(T) = \sum_{n>0} c_{nm} (T^q-T)^{n p^m} + \sum_{i=0}^{m-1}
\sum_{p \nmid n} c_{ni} p^{m-1-i} (T^q-T)^{n p^i} + c_{0m},
\]

\noindent with all $c_{ni} \in R$. As above, assume that $c_{nm} = 0\
\forall n>0$, using Step 1. We show that $c_{ni} \in \Ann_R(p^{m-1} \m)$
by induction on $\deg f$. The base case of $f(T) = c_{0m}$ is obvious.
Now given $f(T) = g(T^q-T)$ as above, notice by the uniqueness of the
representation of $f$ in the preceding equation, the leading degree term
of $f$ comes from a unique summand, say corresponding to $(n,i)$ with
$i<m$. Expanding in powers of $T^q-T$ via computations as above, the
identically zero polynomial
\[
f(T+r) - f(T) = g((T+r)^q - (T+r)) - g(T^q-T) \equiv g(T^q-T-r) -
g(T^q-T) \mod p
\]

\noindent has leading term
\[
p^{m-1-i} c_{ni} \binom{n p^i}{1} (T^q-T)^{n p^i - 1} (-r) =
p^{m-1} \cdot (-n) r c_{ni} (T^q-T)^{n p^i - 1} \mod p^m.
\]

\noindent This term must vanish for all $r \in \m$. As $n \in
(\Z/p^m\Z)^\times$, this shows $c_{ni} \in \Ann_R(p^{m-1} \m)$.
Subtracting $c_{ni} p^{m-1-i}(T^q-T)^{np^i}$ from $f(T)$, the result
holds by induction on $\deg(f)$.
\end{proof}

We conclude our discussion of translation-invariant polynomials by
observing that Theorem \ref{Twitt} helps classify all such polynomials
over a large class of finite commutative rings.
Indeed, given a finite commutative ring $R$, write $R = \times_j R_j$ as
a product of local rings $R_j$, each of which has prime power
characteristic $p_j^{m_j}$, as discussed prior to Theorem \ref{Tmain}.
Now if each $R_j$ is of the form in Theorem \ref{Twitt}, then the
translation-invariant polynomials in $R[T]$ can be classified using the
following result.

\begin{lemma}
Given a commutative ring $R = \times_j R_j$, write $f \in R[T]$ as
$\sum_j f_j$, with $f_j \in R_j[T]$. Then $f$ is translation-invariant
over $R$ if and only if each $f_j$ is translation-invariant over $R_j$.
\end{lemma}

\begin{proof}
Writing $f = \sum_j f_j$ uniquely as above, it follows for $i \neq j$ and
$r_j \in R_j$ that
\[
f_i(T+r_j) - f_i(T) \in r_j R_i[T] = 0.
\]

\noindent This implies the result, since
\[
f \left( T+\sum_j r_j \right) - f(T) = \sum_j \left( f_j(T+r_j) - f_j(T)
\right). \qedhere
\]
\end{proof}





\begin{thebibliography}{10}
\bibitem{BCL}
J.V.~Brawley, L.~Carlitz, and J.~Levine, {\em Power sums of matrices over
  a finite field},
  \href{http://dx.doi.org/10.1215/S0012-7094-74-04102-7}{Duke
  Mathematical Journal} 41(1):9--24, 1974.

\bibitem{Ca1}
L.~Carlitz, {\em Some sums involving polynomials in a Galois field},
  \href{http://dx.doi.org/10.1215/S0012-7094-39-00574-0}{Duke
  Mathematical Journal} 5(4):941--947, 1939.

\bibitem{Ca2}
L.~Carlitz, {\em An analogue of the Staudt--Clausen theorem},
  \href{http://dx.doi.org/10.1215/S0012-7094-40-00703-7}{Duke
  Mathematical Journal} 7(1):62--67, 1940.

\bibitem{Ca3}
L.~Carlitz, {\em The Staudt--Clausen theorem},
  \href{http://dx.doi.org/10.2307/2688488}{Mathematics Magazine}
  34:131--146, 1961.

\bibitem{Cl}
T.~Clausen, {\em Lehrsatz aus einer Abhandlung \"uber die Bernoullischen
  Zahlen},
  \href{http://dx.doi.org/10.1002/asna.18400172204}{Astronomische
  Nachrichten} 17(22):351--352, 1840.

\bibitem{FGO}
P.~Fortuny Ayuso, J.M.~Grau, and A.M.~Oller-Marc\'en, {\em A von
  Staudt-type result for $\displaystyle \sum_{z \in \Z_n[i]} z^k$},
  \href{http://dx.doi.org/10.1007/s00605-015-0736-5}{Monatshefte f\"ur
  Mathematik} 178(3):345--359, 2015.

\bibitem{FGOR}
P.~Fortuny Ayuso, J.M.~Grau, A.M.~Oller-Marc\'en, and I.F.~R\'ua, {\em On
  power sums of matrices over a finite commutative ring}, preprint,
  \href{http://dx.doi.org/10.1142/S0218196717500278}{International
  Journal of Algebra and Computation} 27(5):547--560, 2017.

\bibitem{Ge}
E.-U.~Gekeler, {\em On power sums of polynomials over finite fields},
  \href{http://dx.doi.org/10.1016/0022-314X(88)90023-6}{Journal of Number
  Theory} 30(1):11--26, 1988.

\bibitem{Gos1}
D.~Goss, {\em Von Staudt for ${\bf F}_q[T]$},
  \href{http://dx.doi.org/10.1215/S0012-7094-78-04541-6}{Duke
  Mathematical Journal} 45(4):885--910, 1978.

\bibitem{Gos2}
D.~Goss, {\em Kummer and Herbrand criterion in the theory of function
  fields},
  \href{http://dx.doi.org/10.1215/S0012-7094-82-04923-7}{Duke
  Mathematical Journal} 49(2):377--384, 1982.

\bibitem{Gou}
H.W.~Gould, {\em The Girard--Waring Power Sum Formulas for Symmetric
  Functions, and Fibonacci Sequences},
  \href{http://www.fq.math.ca/37-2.html}{The Fibonacci Quarterly}
  37(2):135--140, 1999.

\bibitem{GO}
J.M.~Grau and A.M.~Oller-Marc\'en, {\em Power sums over finite
  commutative unital rings}, 
  \href{http://dx.doi.org/10.1016/j.ffa.2017.07.003}{Finite Fields and
  their Applications} 48:10--19, 2017.

\bibitem{GOS}
J.M.~Grau, A.M.~Oller-Marc\'en, and J.~Sondow, {\em On the congruence
  $1^m + 2^m + \cdots + m^m \equiv n (\hspace*{-2mm} \mod m)$ with $n |
  m$}, \href{http://dx.doi.org/10.1007/s00605-014-0660-0}{Monatshefte
  f\"ur Mathematik} 177(3):421--436, 2015.

\bibitem{Lu}
E.~Lucas, {\em Th\'eorie des Fonctions Num\'eriques Simplement
  P\'eriodiques} (parts 1,2,3), American Journal of Mathematics 1:
  \href{https://dx.doi.org/10.2307/2369308}{184--196}; 
  \href{https://dx.doi.org/10.2307/2369311}{197--240}; 
  \href{https://dx.doi.org/10.2307/2369373}{289--321}, 1878.

\bibitem{Mor}
P.~Moree, {\em On a theorem of Carlitz--von Staudt},
  C.~R.~Math.~Rep.~Acad.~Sci.~Canada 16(4):166--170, 1994.

\bibitem{St}
K.G.C.~von Staudt, {\em Beweis eines Lehrsatzes, die Bernoullischen
  Zahlen betreffen},
  \href{http://dx.doi.org/10.1515/crll.1840.21.372}{Journal f\"ur die
  reine und angewandte Mathematik} 21:372--374, 1840.

\bibitem{SV}
A.~Stewart and V.~Vologodsky, {\em On the center of the ring of
  differential operators on a smooth variety over $\Z/p^n\Z$},
  \href{http://dx.doi.org/10.1112/S0010437X12000462}{Compositio
  Mathematica} 149(1):63--80, 2013.

\bibitem{Th}
D.S.~Thakur, {\em Power sums of polynomials over finite fields and
  applications: A survey},
  \href{http://dx.doi.org/10.1016/j.ffa.2014.08.004}{Finite Fields and
  Their Applications} 32:171--191, 2015.

\bibitem{Wi}
E.~Witt, {\em Zyklische K\"orper und Algebren der Charakteristik $p$ vom
  Grad $p^n$. Struktur diskret bewerteter perfekter K\"orper mit
  vollkommenem Restklassenk\"orper der Charakteristik $p$},
  \href{http://dx.doi.org/10.1515/crll.1937.176.126}{Journal f\"ur die
  reine und angewandte Mathematik} 176:126--140, 1937.
\end{thebibliography}
\end{document}